\documentclass[11pt,a4paper]{amsart}

\usepackage[utf8]{inputenc}
\usepackage[T1]{fontenc}
\usepackage{amsmath}
\usepackage{amsfonts}
\usepackage{ae}
\usepackage{units}
\usepackage{icomma}
\usepackage{color}
\usepackage{graphicx}
\usepackage{bbm}
\usepackage{amssymb}
\usepackage{dsfont}
\usepackage{cancel}
\usepackage{upgreek}
\usepackage{type1cm}
\usepackage{eso-pic}
\usepackage[breaklinks,pdfpagelabels=false]{hyperref}
\usepackage[hyphenbreaks]{breakurl}
\usepackage{xspace}
\usepackage{verbatim}
\usepackage{enumitem}
\usepackage{bm}
\usepackage[foot]{amsaddr}

\newcommand{\R}{\ensuremath{\mathbbm{R}}}

\newcommand{\E}{\ensuremath{\mathcal{E}}}
\newcommand{\A}{\ensuremath{\mathcal{A}}}
\newcommand{\B}{\ensuremath{\mathcal{B}}}

\newcommand{\abs}[1]{\ensuremath{\left| #1 \right|}}

\newcommand{\noz}{\# 0}
\newcommand{\noo}{\# 1}

\newcommand{\Ekeep}{\ensuremath{{\mathcal E}_{keep}}}
\newcommand{\Egood}{\ensuremath{{\mathcal E}_{good}}}
\newcommand{\Ebad}{\ensuremath{{\mathcal E}_{bad}}}
\newcommand{\indicator}{\ensuremath{\mathds 1}} 
\newcommand{\Exp}{\ensuremath{\mathbbm{E}}}
\newcommand{\ooea}{$(1 + 1)$-EA\xspace}
\newcommand{\hottopic}{\textsc{HotTopic}\xspace}
\newcommand{\Uset}{\ensuremath{\mathcal{U}}}

\renewcommand{\epsilon}{\varepsilon}
\newcommand{\eps}{\varepsilon}

\newtheorem{theorem}{Theorem}[section]

\newtheorem{proposition}[theorem]{Proposition}

\theoremstyle{definition}

\begin{document}

\title[When Does Hillclimbing Fail on Monotone Functions]{When Does Hillclimbing Fail on Monotone Functions: An entropy compression argument}
%Entropy Compression for a Hillclimbing Heuristic on Monotone Functions}
\author{Johannes Lengler$^*$}
\author{Anders Martinsson$^*$}
\author{Angelika Steger$^*$}
\address{$^*$Department of Computer Science, ETH Z\"{u}rich, Switzerland}
\maketitle

\begin{abstract}
Hillclimbing is an essential part of any optimization algorithm. An important benchmark for hillclimbing algorithms on pseudo-Boolean functions $f: \{0,1\}^n \to \R$ are (strictly) \emph{montone} functions, on which a surprising number of hillclimbers fail to be efficient. For example, the $(1+1)$-Evolutionary Algorithm is a standard hillclimber which flips each bit independently with probability $c/n$ in each round. Perhaps surprisingly, this algorithm shows a phase transition: it optimizes any monotone pseudo-boolean function in quasilinear time if $c<1$, but there are monotone functions for which the algorithm needs exponential time if $c>2.2$. But so far it was unclear whether the threshold is at $c=1$.

In this paper we show how Moser's entropy compression argument can be adapted to this situation, that is, we show that a long runtime would allow us to encode the random steps of the algorithm with less bits than their entropy. Thus there exists a $c_0 > 1$ such that for all $0<c\le c_0$ the $(1+1)$-Evolutionary Algorithm with rate $c/n$ finds the optimum in $O(n \log^2 n)$ steps in expectation. 
\end{abstract}

%\keywords{Entropy Compression Method, Evolutionary Algorithm, Monotone Function, Hillclimbing, Mutation Rate, Runtime Analysis}

\section{Introduction}

Hillclimbing is an essential part of any optimization algorithm. The \emph{$(1+1)$-Evolutionary Algorithm} or $(1+1)$-EA is a simple greedy hillclimbing scheme for maximizing a pseudo-Boolean objective function $f:\{0, 1\}^n\rightarrow \mathbb{R}$. We start with a search point $X_0 \in \{0, 1\}^n$ uniformly at random. In the $t$-th round we create an \emph{offspring} $X'$ from the \emph{parent} $X_t$ by flipping each bit of $X_t$ independently with probability $c/n$, where $c$ is the \emph{mutation parameter}. Then we replace the current search point by $X'$ if it has at least the same objective, i.e., we set $X_{t+1} := X'$ if $f(X') \geq f(X_{t})$, and $X_{t+1} := X_t$ otherwise. The phrase $(1+1)$ reflects that in each round the next search point is chosen from one parent plus one offspring. It is clear that, on any function $f$ with a unique global maximum, the $(1+1)$-EA will eventually fixate at at this maximum of $f$.

Here we study the performance of this algorithm on (strictly) monotone functions. A function $f:\{0, 1\}^n\rightarrow\mathbb{R}$ is said to be \emph{monotone}\footnote{We define ``monotone'' in a way that otherwise might rather be called ``strictly monotone'', for ease of terminology. Note that we can't expect efficient runtimes for functions which are monotone in a non-strict sense. For example, we could have $f(x) = 0$ for $x\not\equiv 1$ and $f(1)=1$ otherwise and the search for the optimal solution amounts to searching a needle in a hay stack. Therefore, we define monotone functions in this strict sense in our paper.
} if $f(x)<f(y)$ whenever $x\neq y$ and $x^i\leq y^i$ for all $i\in[n]$, where $x^i$ denote the $i$-th coordinate of $x$. For any such function, the unique global maximum is the all-ones string. Monotone functions are an important class of benchmark functions for hillclimbing schemes, since there exist a large variety of hillclimbing schemes that optimize all monotone functions efficiently. For example, the $(1+1)$ algorithm that creates the offspring by flipping exactly one random bit in each round resembles a coupon collector process, and thus finds the optimum in time $O(n\log n)$.\footnote{This algorithm is called \emph{Random Local Search}. Despite its good behaviour on monotone functions it has severe limitations in practice. For example, other than the \ooea it can't escape local optima and is highly susceptible to any form of noise.} Nevertheless, a surprising number of hillclimbing schemes fail on some monotone functions, see~\cite{lengler2018general} for an overview. 

For the \ooea, while for any constant $c<1$ it is easy to see that the algorithm needs time $O(n\log n)$ to find the optimum of any monotone function~\cite{DoerrJSWZ10}, it was shown in a sequence of papers~\cite{DoerrJSWZ10,doerr2013mutation,lengler2016drift} that for $c > 2.13\ldots$ there are monotone functions (dubbed \hottopic functions in~\cite{lengler2018general}) on which the algorithm needs exponential time. The standard proof techniques for upper runtime bounds fail precisely at $c=1$, and there were split opinions in the community on whether there should be a phase transition from polynomial to exponential at $c=1$~\cite{kotzingpersonal}. On the presumed threshold $c=1$ it follows from a more general model of Jansen~\cite{jansen2007brittleness} that the runtime is $O(n^{3/2})$, but it remained unclear whether the runtime is quasilinear or not. 

The value $c=1$ is of special interest, for several reasons. From a practical perspective, it is considered the standard choice for the mutation parameter and explicitly recommended by textbooks on the subject~\cite{back1996evolutionary, back1997handbook}. From a theoretical perspective, $c=1$ is known to be the optimal parameter choice for linear functions, i.e., for functions of the form $f(X) = \sum_{i=1}^n w_iX^i$, where the $w_i$ are fixed weights. More precisely, the choice $c=1$ gives runtime $(1+o(1)) en\log n$ on any linear function, while any other (constant) choice of $c$ gives a strictly worse leading constant on any linear function~\cite{witt2013tight}.

In this paper we use an entropy compression argument to show that for $c = 1+\eps$ the runtime remains quasilinear. More precisely, we show that a long runtime would allow us to encode the random trajectory of the algorithm with fewer bits than its entropy, which is an information theoretic contradiction. This type of argument is attributed to Moser, who used the technique in his celebrated algorithmic proof of the Lov\'{a}sz local lemma~\cite{moser2009constructive,fortnow2009kolmogorow}. Since then, the method has been used to extend and apply the local lemma~\cite{moser2010constructive,achlioptas2016random}, and used to some extent for colouring problems~\cite{achlioptas2016focused,przybylo2016facial,esperet2013acyclic,grytczuk2013new,dujmovic2016nonrepetitive}. The same idea has been used for analysing optimal data structures, e.g., for cell probing~\cite{brody2015adapt,larsen2012cell} and for sampling~\cite{bringmann2013succinct}. Despite these results, and despite popular blogposts~\cite{tao2009moser,fortnow2009kolmogorow}, the method still does not seem to be widely known outside of these communities.

We use this technique to prove that no phase transition occurs at $c=1$. More precisely, we show the  following.
\begin{theorem}\label{thm:main}
There exists an $\varepsilon>0$ such that, for any (strictly) monotone function $f$ and any constant $0<c\leq 1+\varepsilon$, the \ooea with mutation parameter $c$ requires $O(n \log^2 n)$ steps until it finds the maximum of $f$, and it visits an expected number of $O(n)$ search points. The same remains true if the initial search point of the algorithm is chosen by an adversary.
\end{theorem}

To be more precise, we show that there are constants $\eps, C>0$ such that for all $0<c\leq 1+\eps$ the runtime is at most $C/c \cdot n\log^2 n$, and the number of search points is at most $C \cdot n$.

For the proof, it turns out to be natural to not  measure performance in the number of time steps, but in the number of \emph{updates}, i.e. the number of times $X_{t+1} \neq X_t$. Note that because of the greedy nature of the algorithm, the number of updates coincides with the number of visited search points (minus one).

We first give some intuition on the behaviour of the algorithm and on our proof. For a general monotone function $f$, it is natural to measure the progress of the $(1+1)$-EA is terms of number of one-bits in the current search point. In order to make an update, it is necessary to flip at least one zero-bit into a one-bit, since otherwise the offspring would be rejected due to monotonicity (unless it is identical to the parent, in which case there is no update either). Thus, if the average update does not flip too many ones to zeros, the number of ones in the current search point will tend to $n$ efficiently. For a small mutation parameter (specifically for $c<1$), this is indeed true as the average number of ones flipped to zeros is at most $c$, and this remains true for update steps. For larger $c$, one might expect this to still hold as, intuitively, any offspring with more ones flipped to zeros than zeros flipped to ones should be unlikely to be fitter than its parent. However, the reason this intuition fails for sufficiently large $c$ is that one can ``trick'' the algorithm by weighing a fraction of the remaining zeros much higher than most ones already in the search point. Then, whenever one of these zeroes is flipped, the algorithm will happily keep the offspring, regardless of how many ones are flipped to zeros in the process.

Based on this intuition, we define \emph{good} and \emph{bad} updates. The bad updates capture cases in which a zero-bit with a disproportionally high weight is flipped. Then we study the entropy of the update steps of the algorithms. On the one hand, we will analyze the algorithm in a forward manner to give a lower bound on the entropy of each update. On the other hand, we give a backwards encoding (from last step to first) of the updates steps, and this encoding saves some bits in bad update steps. Since the expected number of bits needed for the encoding is lower bounded by the entropy, we get an upper bound for the expected number of bad update steps. This, in turn, gives us a linear upper bound for the expected number of update steps. Finally, the runtime bound follows by a slight refinement of the calculation, in which we compute how many steps we need to decrease the number of zero-bits from $2^k$ to $2^{k-1}$, for $k=\log n, \dots, 1$.

\section{Preliminaries: properties of single updates}
The idea of this section is to collect properties of single update steps. Throughout this section we will use the following notation. We assume that we are in an (arbitrary, but fixed) state $y\in\{0,1\}^n$. We also assume that $k$ denotes the number of ones in $y$. We denote by $Y'$ the string obtained by flipping each coordinate independently with probability $c/n$. With $\Ekeep$ we denote the event that $f(Y') \ge f(y)$,  corresponding to the event that~$y'$ is accepted as the new state.  We also denote by $U$ the number of bits that are zero in $y$, but one in $Y'$ (upflips) and by $D$ the number of bits that are one in $y$, but zero in $Y'$ (downflips).

In this section we will repeatedly make use of the following fact. 
Let $X$ be a random variable, and $\E$ be some event. With $\indicator_\E$ we denote the indicator variable for the event $\E$. Then by the law of conditional expectation,
\begin{equation}\label{eq:conditional:expectation}
\Exp[X|\E] = {\Exp[X\cdot \indicator_\E]} / {\Pr[\E]}.
\end{equation}

\subsection{Expected number of bits flips}
$\Exp[U+D]$ is easily computed by linearity of expectation to be equal to~$c$. However, we are interested in $\Exp[U+D\mid \Ekeep]$. Observe that $\Pr[\Ekeep]$ is at least the probability that we flip exactly one zero-bit to a one and no other bit. Thus, $\Pr[\Ekeep] \ge (n-k)\frac{c}n(1-\frac{c}n)^{n-1}\geq  \frac{n-k}n ce^{-c}$, where the latter step holds for all $0<c<2$ and $n$ sufficiently large, and follows from the expansion $1-c/n = e^{-c/n + 2c^2/n^2 +O(1/n^3)}$. Observe also that
$\indicator_{\Ekeep} \le \indicator_{\A_1}+\ldots+ \indicator_{\A_{n-k}}$, where
$\A_i$ denotes the event that we flip the $i$-th zero-bit to a one, where $1\le i\le n-k$. Thus \eqref{eq:conditional:expectation} implies that for $0<c<2$, and $n$ sufficiently large,
\begin{equation}\label{eq:UplusD}
\Exp[U+D\mid \Ekeep] 
\le \frac{\sum_{i=1}^{n-k}\Exp[(U+D)\cdot \indicator_{\A_i}]}{\Pr[\Ekeep]}
\le \frac{(n-k)\frac{c}n(1+c)}{ \frac{n-k}n ce^{-c}} = (1+c)e^c.
\end{equation}

\subsection{Change in the number of ones}\label{sec:number:ones}

Our goal in this section is to (lower) bound the change in the number of ones, i.e., to bound
$\Exp[U-D\mid \Ekeep]$.  Clearly,
$\Exp[U \mid \Ekeep] \ge 1$ by monotonicity. To bound $D$, note that it is intuitively clear that $\Exp[D \mid \Ekeep] \leq \Exp[D] \leq c$. A full proof can be found in the appendix. For $c<1$ we thus get from \eqref{eq:conditional:expectation} that $\Exp[U  - D | \Ekeep] \ge 1-c$. Standard drift arguments thus imply that the expected number of updates till  EA reaches the all-ones string is $O(n)$. 

In order to also be able to apply a similar argument for $c > 1$, we need to be more careful. What we will do is to partition updates into good and bad ones, i.e., we let $\Ekeep = \Egood \uplus \Ebad$ and define $\Ebad$ in such a way that  $\Ebad$ happens ``rarely'' {and}  $\Exp[U  - D | \Egood]$ is positive. 

To do so, observe that whenever $U=1$, say bit $i$ is flipped from a zero to a one, we can attribute a \emph{value} to each 1-bit in $y+e_i$ (that is, to the indices $j\in \{i\}\cup \{ a\in[n] : y^a=1\}$) according to $$val_{y+e_i}(j) := f(y+ e_i) - f(y + e_i - e_j),$$
where $e_1, e_2, \dots, e_n$ denotes the standard basis vectors, and $+$ and $-$ denotes vector addition and subtraction respectively. It is natural to think of this value as the ``cost'' of flipping bit $j$ to a zero. Indeed if $Y'$ is obtained from $y$ by flipping bit $i$ from a zero to a one, and bits $j_1, j_2, \dots$ from ones to zeros, then if at least one of the $j$-bits, say $j_1$, has a strictly higher value than $i$, we have
$$f(Y') \leq f(y+ e_i - e_{j_1}) = val_{y+e_i}(i)-val_{y+e_i}(j)<0,$$
which means such a $Y'$ will never be kept. Similarly, if $j_1$ has equal value to $i$, and this is not the only $1$-bit flipped to a zero, then $Y'$ will not be kept.

With this notion at hand we say that an update from $y$ to $Y'$ belongs to \Ebad\ iff 
\begin{enumerate}
\item $Y' \neq y$ and $f(Y') \leq f(y)$, i.e., we actually make an update,
\item there is exactly one zero-bit $i$ that is flipped, i.e, $U=1$, and
\item there exist at least $(1-\alpha) n$ one-bits in $y + e_i$  
whose value is strictly smaller than the value of bit $i$.
\end{enumerate}
To get some intuition behind this definition, observe that it indeed captures cases in which the number of ones may likely {\em decrease}: we only flip one bit from zero to one {\em and} there are many candidate  bits for which we may be able to flip two or more of them  back to zero.  This intuition is formalized by the following proposition.

\begin{proposition}\label{prop:excex} For any $0 \leq \alpha \leq \frac{1}{2}$ and $0 \leq c\leq 1/(1-\alpha)$, we have for all $n \geq 3$
\begin{align*} 
\mathbb{E}[U-D \mid \Ebad] &\geq 1-c,\\
\mathbb{E}[U-D \mid\Egood] &\geq \left(1-\frac{c}{n}\right)^{\alpha n} \left( 1-(1-\alpha)c\right) \geq e^{-2\alpha c} \left( 1-(1-\alpha)c\right).
\end{align*}
\end{proposition}
\begin{proof}
Let $\mathcal{B}_i$ denote the event that the $i$-th one-bit in $y$ gets flipped in $Y'$. Then, by linearity of expectation, we have
$\mathbb{E}[U-D\mid \Ebad] \geq 1-\sum_{i=1}^k\Pr[\mathcal{B}_i\mid \Ebad].$ By Bayes' Theorem, we have $\Pr[\mathcal{B}_i\mid \Ebad] = \frac{c}{n}\cdot\frac{\Pr[\Ebad\mid \mathcal{B}_i]}{\Pr[\Ebad]}\leq \frac{c}{n},$ where the last step follows by a simple coupling argument, similar as in Section~\ref{sec:number:ones}. Hence $\mathbb{E}[U-D\mid \Ebad] \geq 1-c\frac{k}{n}\geq 1-c,$
as desired.

As for $\Egood$, the second inequality follows from $1-(1-\alpha)c \geq 0$, since $1-x \geq e^{-2x}$ for all $0\leq x \leq 2/3$. 
For the first inequality, let $\Uset$ denote the set of indices of zero-bits in $y$ that get flipped to one-bits in $Y'$. Then
$$\mathbb{E}[U-D\mid\Egood] = \sum_A \Pr[\mathcal{U}=A\mid\Egood] \cdot \mathbb{E}[U-D\mid \Egood\cap\{\mathcal{U}=A\}].$$
Thus, it suffices to estimate $\mathbb{E}[U-D\mid \Egood\cap\{\mathcal{U}=A\}]$ for any set $A\subset[n]$ such that $\Pr[\mathcal{U}=A\vert\Egood]$ is non-zero.

If $\abs{A}\geq 2$, the same argument as for $\Ebad$ gives $$\mathbb{E}[U-D\mid  \Egood\cap\{\mathcal{U}=A\}] \geq 2-c.$$ It remains to consider the case of $\abs{A}=1$, say $A=\{i\}$. In this case, let $k'=\min(k, \lfloor (1-\alpha) n\rfloor)$, and order the one-bits in $y$, $j_1, j_2, \dots j_k$, in descending order with respect to $val_{y+e_i}(j)$ with ties broken arbitrarily. In order for $\mathcal{U}=A$ to be compatible with a good update, we can assume that the values of $j_1, j_2, \dots j_{k-k'}$ must be greater than or equal to the value of $i$.

With these definitions at hand, we write
$$\mathbb{E}[U-D\mid  \Egood\cap\{\Uset=A\}] = \frac{1}{\Pr[\Egood\mid\Uset=A]}\mathbb{E}[\mathbbm{1}_{\Egood}(U-D)\mid\Uset=A].$$
Note that, conditioned on $\Uset=A$, $\mathbbm{1}_{\Egood}(U-D)=0$ whenever one of the bits $j_1, \dots j_{k-k'}$ are flipped. This is because this is either the only bit flipped to a zero, in which case $U-D=1-1=0$, or one additional bit is flipped to a zero, in which case $f(Y')<f(y)$. Whenever the bits $j_1, \dots j_{k-k'}$ remain ones, on the other hand, we can lower bound $\mathbbm{1}_{\Egood}(U-D)$ by one minus the number of bits among $j_{k-k'+1}, \dots j_{k}$ that get flipped. Thus
$$\mathbb{E}[\mathbbm{1}_{\Egood}(U-D)\mid\Uset=A] \geq (1-\frac{c}{n})^{k-k'}(1-\frac{c}{n}k').$$
By assumption, we have $1-\frac{c}{n}k' \geq 1-(1-\alpha)c \geq 0$ which means that
$$\mathbb{E}[U-D\mid \Egood\cap\{\Uset=A\}] \geq \mathbb{E}[\mathbbm{1}_{\Egood}(U-D)\mid\Uset=A] \geq (1-\frac{c}{n})^{\alpha n}(1-(1-\alpha)c),$$
as desired. The proposition follows by observing that $2-c \geq 1-(1-\alpha)c$ as $2-c-1+(1-\alpha)c = 1-\alpha c \geq 1-\frac{\alpha}{1-\alpha} \geq 0,$ where in the second to last step we used $c \leq 1/(1-\alpha)$.

\end{proof}

Note that for any fixed $0< \alpha \leq 1/2$ the considered range in Proposition~\ref{prop:excex} also contains some values $c>1$. Moreover, for the considered range the lower bound for $\mathbb{E}[U-D\mid \Egood] $ is positive. If we could thus show that $\Ebad$ occurs only sufficiently rarely, then we might hope to be able to bound the number of updates. This is what we will do with the entropy compression argument.

\subsection{Entropy of an update step} In this section we study the entropy of a single update starting from a fixed state $y$. We refer the reader who is not familiar with information theory to the introduction in~\cite{cover2012elements}. Naturally, the entropy will depend on $y$. Recall that the random variables $U$ and $D$ denote the number of upflips and downflips, respectively.

 \begin{proposition}\label{prop:entropy:update}
 For any $c < 4/3$, any $0\le k < n$ and any $y\in\{0, 1\}^n$  with exactly $k$ ones we have
$$\mathbb{H}(Y'\mid \Ekeep) \geq \Exp[\log_2\left({n-k\choose U}{k+U \choose D}\right) \mid \Ekeep ],$$
where $\mathbb{H}(Y'\mid \Ekeep)$ denotes the binary entropy of the conditional distribution of $Y'$ given $\Ekeep.$ 
\end{proposition}
\begin{proof} 
Let $\A_{u,d}$ denote the set of all strings $z\in\{0,1\}^n$ such that $f(z)\ge f(y)$ and such that $z$ can be obtained from $y$ by flipping precisely $u$ zeros to ones and $d$ ones to zeros. 
Then
$$
p_{keep}  := \Pr[\Ekeep] = \sum_{u=1}^{n-k}\sum_{d= 0}^k \sum_{z\in \A_{u,d}} (\frac{c}n)^{u+d}(1-\frac{c}n)^{n-u-d}.
$$
To simplify notation we write
$$q_{u,d}:=\abs{\mathcal{A}_{ud}}\left(\frac{c}{n}\right)^{u+d}\left(1-\frac{c}{n}\right)^{n-u-d}$$
and 
$$a_{u,d}:= \left(\frac{c}{n}\right)^{u+d}\left(1-\frac{c}{n}\right)^{n-u-d}.$$
By the definition of entropy we have
\begin{align*}
\mathbb{H}(Y' \mid  \Ekeep )&=-\frac{1}{p_{keep}}\sum_{u=1}^{n-k}\sum_{d= 0}^k  \sum_{z\in \A_{u,d}} a_{ud} \log_2(\frac{a_{u,d}}{p_{keep}})\\
&=\frac{1}{p_{keep}}\sum_{u=1}^{n-k}\sum_{d= 0}^k  q_{ud} \log_2(\frac{p_{keep}}{a_{u,d}}).
\end{align*}
We want to show that this is at least as large as the expectation in the statement of the proposition. With 
$$b_{u,d}:={n-k \choose u}{k+u \choose d}$$
we can write this expectation as
$$\Exp[\log_2\left({n-k\choose U}{k+U \choose D}\right) \mid \Ekeep ] 
=\frac{1}{p_{keep}}\sum_{u=1}^{n-k}\sum_{d=0}^k q_{u,d}\log_2\left(b_{u,d}\right).$$
Hence the proposition follows if we can show that the difference is non-negative. That is, we have to show that
\begin{equation*}
\frac1{p_{keep}}\sum_{u=1}^{n-k}\sum_{d=0}^k q_{u,d} \log_2\left(\frac{p_{keep}}{a_{u,d}b_{u,d}}\right) \stackrel!\ge 0.
\end{equation*}
Multiplying by $p_{keep} \ge 0$ and partitioning the log amounts to showing that
\begin{equation}\label{eq:Hdiff}
\Delta := p_{keep}\log_2 p_{keep}+\sum_{u=1}^{n-k}\sum_{d=0}^k q_{ud}\log_2\left(\frac{1}{a_{u,d}b_{u,d}}\right) \stackrel!\ge 0.
\end{equation}
To this end, an elementary calculation shows that the product $a_{u,d}b_{u,d}$ is either maximized by the case $u=1$, $d=0$, or by $u=d=1$. We defer the calculation to the appendix.

Assume first that $a_{1,1}b_{1,1} \leq a_{1,0}b_{1,0}$. Then the left hand side in \eqref{eq:Hdiff} satisfies
$$
\Delta \ge p_{keep}\log_2 p_{keep}+p_{keep}\log_2(\frac{1}{a_{1,0}b_{1,0}}) 
$$
which is non-negative, as $p_{keep} \ge q_{1,0}=a_{1,0}b_{1,0}$.
For the other case, assume $a_{1,1}b_{1,1} > a_{1,0}b_{1,0}$. Writing $p_{keep} = (1+x)q_{1,0}$ for some $x\ge 0$ 
and recalling that $q_{1,0} = a_{1,0}b_{1,0}$ we obtain
\begin{align*}
\Delta &\ge (1+x)q_{1,0}\cdot(\log_2((1+x) q_{1,0}) +q_{1,0}\log_2(\frac{1}{q_{1,0}}) + xq_{1,0}\log_2(\frac{1}{a_{1,1}b_{1,1}})\\
&=q_{1,0}\cdot\left( (1+x)\log_2(1+x) + x \log_2\frac{a_{1,0}b_{1,0}}{a_{1,1}b_{1,1}}\right).
\end{align*}
Now observe that $\frac{a_{1,0}b_{1,0}}{a_{1,1}b_{1,1}} = \frac{1-c/n}c \frac{n}{k+1} > 3/4$, for $n$ sufficiently large. The claim now follows from $(1+x)\log_2(1+x) +x \log_2(3/4) \ge 0$ for all $x\ge 0$.
\end{proof}

\section{Entropy of the Markov chain}
The aim of this section is to provide bounds on the entropy of the sequence of updates in the $(1+1)$-EA. Given the sequence $(X_t)_{t=0}^\infty$ of search points, as generated by the algorithm when started in some fixed state $X_0=x\in\{0, 1\}^n$, we define $T$ as the number of updates, that is, the number of times $t=0, 1, \dots$ such that $X_t\neq X_{t+1}$. For each $t=0, 1, \dots T$, we denote by $Y_t$ the state of the algorithm after $t$ update steps. To simplify notation later on, we want $Y_t$ be defined for all $t\geq 0$, so we define $Y_{t}$ to be the all-one string if $t > T$. We note that as $(X_t)_{t=0}^\infty$ is a Markov chain, so is $(Y_t)_{t=0}^\infty$. More precisely, the transition probabilities of  $(Y_t)_{t=0}^\infty$ are the ones obtained from transition probabilities of $(X_t)_{t=0}^\infty$ by removing self-transitions from all states besides the all ones state. 

We will use $\noz_t$, $\noo_t$ to denote the number of zero-bits and one-bits in $Y_t$, respectively. Furthermore, we use $U_t$ and $D_t$ to denote the number of upflips (zero-to-one) and downflips (one-to-zero) from $Y_t$ and $Y_{t+1}$, respectively.

Recall that the entropy of (the trajectory of) the Markov chain $(Y_t)_{t=0}^{\infty}$, as described above, can be written as
$$\mathbb{H}((Y_t)_{t=0}^\infty) = 
\sum_{t=0}^\infty \mathbb{H}(Y_{t+1} \mid Y_t),$$
where
 $$
 \mathbb{H}(Y_{t+1} \mid Y_t) =-\sum_{y}\Pr[Y_t=y] \mathbb{H}(Y_{t+1} \mid Y_t=y).$$
This entropy is finite almost surely, since the Markov chain will converge to the absorbing state almost surely. The next two propositions bound this entropy from below and  above.
 \begin{proposition}\label{entropy:lower}
$$\mathbb{H}((Y_t)_{t=0}^{\infty}) \ge \Exp\left[\sum_{t=0}^{T-1}\log_2\left({\noz_t\choose U_{t}}{\noo_t+U_t \choose D_{t}}\right)\right].
$$
\end{proposition}
 \begin{proof}
Observe that, for any $y$, the term  $\mathbb{H}(Y_{t+1} \mid Y_t=y)$ does not depend on $t$, as $(Y_t)_{t=0}^{\infty}$ is a time-homogenous Markov chain. If we thus let $g(y):=\mathbb{H}(Y_{t+1}\mid  Y_t=y)$,  we get 
$$\mathbb{H}(Y_{t+1} \mid Y_t) =\Exp[g(Y_t)]$$
and thus
$$\mathbb{H}((Y_t)_{t=0}^\infty) \geq \sum_{t=0}^\infty \mathbb{E}\left[g(Y_t)\right] = \mathbb{E}\left[\sum_{t=0}^\infty g(Y_t)\right] = \mathbb{E}\left[\sum_{t=0}^{T-1} g(Y_t)\right],$$
where in the last step we have used that the conditional entropy of an update is zero once we have reached the all-ones state. 
As Proposition~\ref{prop:entropy:update} implies that 
$\Exp[g(Y_t)] \ge \Exp[\log_2\left({\noz_t\choose U_{t}}{\noo_t+U_t \choose D_{t}}\right)]$, the proposition follows.
\end{proof}

\begin{proposition}\label{entropy:upper}
There exists a constant $C$ such that for all $0<\alpha<1$ and for all $0<c < 2$, 
$$\mathbb{H}((Y_t)_{t=0}^{\infty}) \le C\Exp[T] -\log_2(1/\alpha) \Exp[T_{bad}] + \Exp\left[\sum_{t=0}^{T-1}\log_2\left({\noz_{t+1}\choose D_{t}}{\noo_{t+1} +D_t\choose U_{t}}\right) \right]\!,$$
where $T_{bad}$ denotes the number of bad updates  (as defined in Section~\ref{sec:number:ones}). 
\end{proposition}
\begin{proof}
Recall that the entropy represents a lower bound on the expected number of bits needed to represent  all information of the process. Thus, the expected length of any encoding strategy for the traces of the chain will form an upper bound on the entropy. We proceed as follows. We encode the process \emph{backwards}, i.e.\ we start from the all-ones vector that is the unique absorbing state of the process. For each update we encode
\begin{list}{}{\topsep0pt\itemsep0pt\leftmargin0.8cm\labelsep0.2cm\labelwidth0.6cm}
\item[$(i)$\hfill]whether the update is good or bad,
\item[$(ii)$\hfill]the number of downflips $D_t$ and the number of upflips $U_t$,
\item[$(iii)$\hfill]the actual choice of which $D_t$ zero-bits in $Y_{t+1}$ were the bits that were flipped from one to zero in the update from $Y_t$ to $Y_{t+1}$, and
\item[$(iv)$\hfill]the actual choice of which $U_t$ one-bits in $Y_{t+1}$ were the bits that were flipped from zero to one.
\end{list} 
To mark the global end of our encoding, we are somewhat wasteful: we start the encoding of each update with a one-bit and conclude the whole encoding with a single zero-bit.
For each update  we  encode $(i)$ with a single bit, and $D_t$ and $U_t$ with $D_t+U_t+2$ bits using a unary encoding. 
From~\eqref{eq:UplusD} we deduce  that there exists a constant $C > 0$ such the expected length of the encoding of $(i)$ and $(ii)$ for all updates is bounded by $C\Exp[T]$. For the encoding of $(iii)$ observe that the $D_t$ bits that correspond to downflips have to be chosen from the zero-bits in $Y_{t+1}$. We thus can encode the actual choice by $\lceil\log_2{\noz_{t+1}\choose D_t} \rceil$ bits. For the encoding in $(iv)$
we distinguish between good and bad updates. For a good update we proceed similarly as in $(iii)$. As the $D_t$ bits that correspond to upflips have to be chosen from one-bits in $Y_{t+1}$, we can encode the actual choice by $\lceil\log_2{\noo_{t+1}\choose U_t} \rceil$ bits. For a bad update we can be more efficient in $(iv)$. Observe first that a bad update implies $U_t=1$. We thus need to specify only a single bit. Recall also that the definition of bad updates implies that 
there exist at least $(1-\alpha)n$ one-bits in $Y_t$ that have lower value than the bit that we want to flip. As the number of one-bits in $Y_t$ is bounded by $\noo_{t+1}+D_t$ we thus see that we have 
at most $\noo_{t+1}+D_t - (1-\alpha)n \le \alpha(\noo_{t+1}+D_t )$ bits from which we can choose the bit that corresponds to the (single) upflip. We can thus encode the choice of this bit with
$\lceil\log_2(\alpha( \noo_{t+1}+D_t)) \rceil $ bits, which is less than $
\lceil\log_2{\noo_{t+1}+D_t\choose U_t} \rceil -\log_2(1/\alpha)+1$ bits.
The claimed bound in the proposition follows by collecting all terms.
\end{proof}

\section{Proof of the theorem}

We first obtain a bound on the expected number of bad updates by comparing upper and lower bound on the entropy of the Markov chain.
\begin{proposition}\label{bound:bad}
If the algorithm starts in a state with exactly $k$ ones, then
$$\Exp[T_{bad}] \le {\textstyle\frac{C}{\log_2(1/\alpha)}}\Exp[T] + {\textstyle\frac{1}{\log_2(1/\alpha)}} {\log_2{n\choose k}},$$
where $C$ is the constant from Proposition~\ref{entropy:upper}.
\end{proposition} 
\begin{proof}
Collecting and rearranging the terms from Propositions~\ref{entropy:lower} and~\ref{entropy:upper} we get
\begin{align}\label{eq:telescoping}
\Exp[T_{bad}] \le {\textstyle\frac{C}{\log_2(1/\alpha)}}\Exp[T] + {\textstyle\frac{1}{\log_2(1/\alpha)}} \Exp\left[\log_2\left(\prod_{t=0}^{T-1}
\frac{{\noz_{t+1}\choose D_{t}}{\noo_{t+1}+D_t \choose U_{t}}}{{\noz_t\choose U_{t}}{\noo_t+U_t \choose D_{t}}}\right)\right].
\end{align}
Using the formulas $\noo_{t+1}=\noo_t+U_t-D_t$ and $\noz_{t+1}=\noz_t-U_t+D_t$, it is easy to see that, for any $0\leq t\leq T-1$,
\begin{align*}
\frac{{\noz_{t+1}\choose D_{t}}{\noo_{t+1}+D_t \choose U_{t}}}{{\noz_t\choose U_{t}}{\noo_t+U_t \choose D_{t}}}&=\frac{\noz_{t+1}!}{\noz_t!}\cdot\frac{\noo_{t+1}!}{\noo_t!}.
\end{align*}
Hence, the product in~\eqref{eq:telescoping} is telescoping, and we  get
$$
\Exp[T_{bad}] \le {\textstyle\frac{C}{\log_2(1/\alpha)}}\Exp[T] + {\textstyle\frac{1}{\log_2(1/\alpha)}} \Exp\left[\log_2\left(
\frac{\noz_{T}! \cdot \noo_T!}{\noz_0!\cdot\noo_0!}\right)\right].$$
The claim now follows  from
$
\frac{\noz_{T}! \cdot \noo_T!}{\noz_0!\cdot\noo_0!} = {{n\choose \noo_0}}/{{n\choose \noo_T}} = {n\choose \noo_0}
$, as $\noo_0=k$ and $\noo_T=n$.
\end{proof}

From this and Proposition~\ref{prop:excex} we obtain an upper bound on the number of updates. 
\begin{proposition}\label{bound:total}
There exists an $\eps >0 $ and $\beta >0$ such that for any $0<c\le 1+\eps$ the following holds.
If the algorithm starts in a state with exactly $k$ ones, then
$$\Exp[T] \le \beta (n-k) + \beta {\log_2{n\choose k}}.$$
\end{proposition} 
\begin{proof}
Observe that $\sum_{t=0}^\infty(\noo_{t+1}-\noo_{t}) = n-\noo_0$. If we thus start in a state with exactly $k$ ones, then whenever $0 \leq \alpha\leq 1/2$ and $0 < c \leq 1/(1-\alpha)$ we get from linearity of expectation and  Proposition \ref{prop:excex} that 
\begin{align}\label{eq:excex}
n -k &= \sum_{t=0}^\infty\mathbb{E}[\noo_{t+1}-\noo_{t}] \\ 
&= \sum_{t=0}^\infty\left(\mathbb{E}[\noo_{t+1}-\noo_{t} | {\Egood(t)}] \cdot \Pr[{\Egood(t)}] +\right.\nonumber\\
&\hspace*{2.5cm}\left.\mathbb{E}[\noo_{t+1}-\noo_{t} | {\Ebad(t)}] \cdot \Pr[{\Ebad(t)}] \right)\nonumber\\
&\stackrel{P\ref{prop:excex} }\geq e^{-2\alpha c} (1-(1-\alpha)c)\cdot \Exp[T-T_{bad}] + (1-c)\cdot \Exp[T_{bad}],\nonumber
\end{align} 
where we have used $\mathcal{E}_{good/bad}(t)$ to denote the event that $Y_t$ is not the all-ones string and the update from $Y_t$ to $Y_{t+1}$ is bad or good, respectively. (Note that once  the Markov chains has reached the all-ones state,   neither $\Ebad(t)$ nor $\Egood(t)$ can occur and the contribution of the corresponding term in the last sum is  zero, as is desired.)

For ease of notation let $D := e^{-2\alpha c} (1-(1-\alpha)c)$. Then the above can be rewritten as
 $$n-k\ge D\Exp[T] - (D-1+c) \Exp[T_{bad}].$$ 
Let $C$ be the constant from Proposition~\ref{bound:bad}. We may assume $C>1/2$. For a fixed $0<\alpha <1$ (to be determined later) we will choose an $\eps >0$ such that $1+\eps < 1/(1-\alpha)$. Then~\eqref{eq:excex} holds for all $c \in [0,1+\eps]$, and we have $D>0$ for any such $c$. Moreover, since $D$ is a continuous function, it attains a minimum $D_{\min} = D_{\min}(\alpha)>0$ on the compact interval $c \in [0,1+\eps]$. We may assume that $\eps \leq 1/3$ and $\eps < D_{\min}/2$, and by Proposition~\ref{bound:bad}, for all $0<c \leq 1+\eps$,
$$
n-k \ge 
\left(D - \frac{C}{\log_2(1/\alpha)}(D-1+c) \right) \cdot \Exp[T] - \frac{D-1+c}{\log_2(1/\alpha)} \log_2{n\choose k}.
$$
Set now $\alpha=2^{-2C} <1/2$ and observe that then the term in front of $\Exp[T]$ is equal to
$\frac12(D+1-c) \geq \frac14 D_{\min}$; the claim of the proposition follows.\end{proof}

\begin{proof}[Proof of Theorem~\ref{thm:main}]
As $\log_2{n\choose k}\le n$, the claim on the number of search points follows immediately from Proposition~\ref{bound:total}. For the bound on the number of steps we have to be more careful, as we also have to count the number of
steps \emph{between} updates. To do so the following observation is useful. Suppose we are in a state with exactly $k$ ones. Then the probability that we flip exactly one zero-bit in the next step is $(n-k)\frac{c}n(1-\frac{c}n)^{n-k} \geq \tfrac12 ce^{-c}\frac{n-k}n$, which holds for $n$ sufficiently large. Note that we may assume $c < 4/3$, in which case we obtain a probability of at least $\tfrac1{10} c\frac{n-k}n$. As we will accept any of these moves, we thus see that the expected number of steps until the next update is at most $\tfrac{10}{c}\cdot n/(n-k)$. 
 This together with  Proposition~\ref{bound:total} implies a bound of $O(n^2/c)$ on the running time of the \ooea. To get a quasilinear bound we partition  the trace of the algorithm in phases. 
For this, let $S_k \subseteq\{0,1\}^n$ denote the set of strings with at most~$2^k-1$ zeros, where $0\le k \le \lfloor\log_2 n\rfloor-1 =: k_0$. 

Before reaching a state from $S_{k_0}$ the expected time between two updates is just $O(1/c)$ (as then we still have a constant fraction of zeros to choose from). From Proposition~\ref{bound:total} we thus know that the Markov chain $(X_t)_{t=0}^\infty$ will reach a state from $S_{k_0}$ in $O(n/c)$ steps. 
Next we consider the phases in which we start in a state from $S_{k}$ and terminate (the phase) when we reach a state from $S_{k-1}$ (for the first time). Denote by $T_k$ the number of update steps in this phase. We can use 
the bound from Proposition~\ref{bound:total}  (that considers the run of the Markov chain until it reaches the all-ones string) to obtain 
$$
\Exp[T_k] \le \beta (n-(n-2^{k}+1)) + \beta \log_2 {n\choose n-2^{k}+1} \le 2 \beta 2^{k}\log n. 
$$
As we argued above, in this phase   the expected number of steps between updates is bounded by $O(n/(2^{k}c))$ (as we always have at least $2^{k-1}$ zeros). The expected number of steps in this phase is thus bounded by $O(n\log n /c)$, where the hidden constant holds uniformly for all phases. Since we assumed $c$ to be a positive constant, the expected number of steps per phase is $O(n\log n)$. As the number of phases is $O(\log n)$, the theorem follows.
\end{proof}

\newpage

\appendix

\section{Proof that $\Exp[D \mid \Ekeep] \leq \Exp[D]$.}
In this section we prove that $\Exp[D \mid \Ekeep] \leq \Exp[D]$, which is used in Section~\ref{sec:number:ones}. For a one-bit $j$, let $\B_j$ denote the event that this bit is flipped into zero. We can divide all potential offspring into pairs $y_0$, $y_1$ which agree in all bits except that $y_0^j = 0$, but $y_1^j=1$. Note that $\Pr[Y' = y_0 \mid \B_j] = \Pr[Y' = y_1 \mid \neg \B_j]$, because all bits are flipped independently. Moreover, since $f(y_0) < f(y_1)$, we have the implication ``$y_0$ is accepted $\implies$ $y_1$ is accepted''. Hence, 
\begin{align}\label{eq:condition_on_flip}
\Pr[\Ekeep \mid \B_j] & = \sum_{y_0 \in \{0,1\}^n, y_0^{j}=0} \Pr[Y' = y_0] \cdot \indicator_{\Ekeep}(y_0) \nonumber \\
& \leq \sum_{y_1 \in \{0,1\}^n, y_1^{j}=1} \Pr[Y' = y_1] \cdot \indicator_{\Ekeep}(y_1) = \Pr[\Ekeep \mid \neg \B_j],
\end{align}
and therefore, 
\begin{align*}
\Pr[\Ekeep] & = \Pr[\Ekeep \mid \B_j] \cdot \Pr[\B_j] + \Pr[\Ekeep \mid \neg\B_j] \cdot \Pr[ \neg\B_j] \\
 &\stackrel{\eqref{eq:condition_on_flip}}{\geq}  \Pr[\Ekeep \mid \B_j] \cdot \Pr[\B_j] + \Pr[\Ekeep \mid \B_j] \cdot (1-\Pr[\B_j]) = \Pr[\Ekeep \mid \B_j].
\end{align*}
In particular, this implies
\[
\Pr[\B_j \mid \Ekeep] = \frac{\Pr[\Ekeep \mid \B_j] \Pr[\B_j]}{\Pr[\Ekeep]} \leq \Pr[\B_j].
\]
Since this holds for all one-bits $j$, and since $D = \sum_j \indicator_{\B_j}$, where the sum runs over all one-bits $j$, we obtain $\Exp[D \mid \Ekeep] \leq \Exp[D] \leq c$.

\section{Missing details in the proof of Proposition~\ref{prop:entropy:update}}
Here we show that for $c<4/3$ and $n$ large enough, the product $a_{u,d}b_{u,d}$ is maximized either for $u=1$, $d=0$, or for $u=d=1$. As a reminder, we repeat the definitions of $a_{u,d}$ and $b_{u,d}$.
\begin{align*}
a_{u,d} & := \left(\frac{c}{n}\right)^{u+d}\left(1-\frac{c}{n}\right)^{n-u-d}.\\
b_{u,d} &:={n-k \choose u}{k+u \choose d}
\end{align*}

We first observe $a_{u,d} / a_{u-1,d-1} < \frac{17}9 \frac1{n^2}$ for all $u,d \geq 1$. Thus
$$
\frac{b_{u,d}}{b_{u-1,d-1}} = \frac{{n-k\choose u}}{{n-k\choose u-1}} \cdot\frac{{k+u\choose d}}{{k+u-1\choose d-1}} = \frac{n-k-u+1}{u}\cdot  \frac{k+u-d+1}{d} \le \frac{n^2}{ud}
$$
implies that $a_{u,d}b_{u,d} \le a_{u-1,d-1}b_{u-1,d-1}$ for all $u,d\ge 2$. Similarly, since for $u \geq 1$,
$$
\frac{b_{u,1}}{b_{u-1,1}} =\frac{{n-k\choose u}}{{n-k\choose u-1}}  \cdot\frac{{k+u\choose 1}}{{k+u-1\choose 1}}    =  \frac{n-k-u+1}{u}\cdot \frac{{k+u}}{{k+u-1}} \stackrel{k\ge d=1}\le \frac{3}{2u} \cdot n  
$$
we deduce that $a_{u,1}b_{u,1} \le a_{1,1}b_{1,1}$ for all  $u\ge 2$ and from 
$$
\frac{b_{1,d}}{b_{1,d-1}} = \frac{{k+1\choose d}}{{k+1\choose d-1}} =  \frac{k+1-d+1}{d}\le\frac{n}d
$$
we get $a_{1,d}b_{1,d} \le a_{1,1}b_{1,1}$ for all  $d\ge 2$. Together, this gives $a_{u,d}b_{u,d} \le a_{1,1}b_{1,1}$  for all $u,d \ge 1$. Finally, for $u \geq 2$ and $n$ sufficiently large,
$$
\frac{a_{u,0}b_{u,0}}{a_{u-1,0}b_{u-1,0}} =\frac{\frac{c}{n}}{(1-\tfrac{c}{n})}\cdot \frac{{n-k\choose u}}{{n-k\choose u-1}}    =  \frac{c}{n-c}\cdot \frac{n-k-u+1}{u} \le \frac{n}{n-c}\cdot \frac{c}{u} \leq 1,
$$
so $a_{u,0}b_{u,0} \le a_{1,0}b_{1,0}$  for all $u\ge 2$. Altogether, we thus have 
$$
a_{u,d}b_{u,d} \le \max\{a_{1,0}b_{1,0}, a_{1,1}b_{1,1}\} \qquad\text{for all $u\ge 1, d\ge 0$},
$$
as required.

\bibliographystyle{abbrv}
\bibliography{references}
\end{document}